\documentclass[bj]{imsart}

\RequirePackage[OT1]{fontenc}
\RequirePackage{amsfonts,amsthm,amsmath,amssymb}
\RequirePackage[colorlinks,citecolor=blue,urlcolor=blue]{hyperref}

\arxiv{arXiv:1706.09298}

\renewcommand{\P}{\mathbb{P}}
\newcommand{\F}{\mathcal{F}}
\newcommand{\E}{\mathbb{E}}
\newcommand{\Real}{\mathbb{R}}
\newcommand{\ctg}{\mathrm{ctg}}
\newcommand{\Res}{\mathrm{Res}}
\renewcommand{\Re}{\mathrm{Re}}
\renewcommand{\Im}{\mathrm{Im}}

\startlocaldefs
\numberwithin{equation}{section}
\theoremstyle{plain}
\newtheorem{thm}{Theorem}[section]
\newtheorem{lem}{Lemma}[section]
\theoremstyle{remark}
\newtheorem{rem}{Remark}[section]
\endlocaldefs

\begin{document}

\begin{frontmatter}
\title{On the eigenproblem for Gaussian bridges}
\runtitle{On the eigenproblem for Gaussian bridges}

\begin{aug}
\author{\fnms{Pavel} \snm{Chigansky}\thanksref{a}\ead[label=e1]{pchiga@mscc.huji.ac.il}}
\author{\fnms{Marina} \snm{Kleptsyna}\thanksref{b,e2}\ead[label=e2,mark]{marina.kleptsyna@univ-lemans.fr}}
\and
\author{\fnms{Dmytro} \snm{Marushkevych}\thanksref{b,e3}\ead[label=e3,mark]{dmytro.marushkevych.etu@univ-lemans.fr}}

\address[a]{Department of Statistics,
The Hebrew University,
Mount Scopus, Jerusalem 91905,
Israel.
\printead{e1}}

\address[b]{Laboratoire de Statistique et Processus,
Le Mans Universit\'{e},
France.
\printead{e2,e3}
}

\runauthor{P. Chigansky, M. Kleptsyna and D. Marushkevych }

\affiliation{The Hebrew University and Le Mans Universit\'{e}}

\end{aug}

\begin{abstract}
Spectral decomposition of the covariance operator is one of the main building blocks in the theory 
and applications of Gaussian processes. Unfortunately it is notoriously hard to derive in a closed form. 
In this paper we consider the eigenproblem for Gaussian bridges. 
Given a {\em base} process, its bridge is obtained by conditioning the trajectories to start and 
terminate at the given points. What can be said about the spectrum of a bridge, given the spectrum 
of its base process? We show how this question can be answered asymptotically for a family of processes, 
including the fractional Brownian motion.
\end{abstract}


\end{frontmatter}

\section{Introduction}

The eigenproblem for a centered process $X=(X_t, t\in [0,1])$ on a probability space $(\Omega, \F,\P)$ 
consists of finding all pairs $(\lambda,\varphi)$ satisfying the equation
\begin{equation}\label{eig}
\int_0^1 K(s,t) \varphi(s)ds = \lambda \varphi(t), \quad t\in [0,1],
\end{equation}
where $K(s,t)=\E X_sX_t$ is the covariance function of $X$. If $K$ is square integrable, 
this problem is well known to have countably many solutions:
the eigenvalues $\lambda_n$, $n\in \mathbb{N}$ are nonnegative and converge to zero, when put in the decreasing order, and
the corresponding eigenfunctions $\varphi_n$ form an orthonormal basis in $L^2([0,1])$. 

One of the earliest and most influential implications of this result is the Karhunen--Lo\'{e}ve theorem, which asserts 
that $X$ admits the representation as the $L^2(\Omega)$-convergent series 
\begin{equation}\label{KLexp}
X_t = \sum_{n=1}^\infty \langle X,\varphi_n \rangle \varphi_n(t)
\end{equation}
where the scalar products $\langle X,\varphi_n \rangle =\int_0^1 X_s \varphi_n(s)ds$ are orthogonal zero mean 
random variables with variance $\E \langle X,\varphi_n \rangle^2 = \lambda_n$. 

Spectral decomposition \eqref{KLexp} is useful in both theory and applications (see, e.g., \cite{BTA04}, \cite{L12}).
However explicit solutions to eigenproblem \eqref{eig} are notoriously hard to find and they are available only 
in special cases  \cite{Istas06}, \cite{DM03, DM08}, \cite{N09b}, \cite{NN04ptrf, NN04tpa},  
including the Brownian motion with $K(s,t)=t\wedge s$: 
\begin{equation}\label{eigBm}
\lambda_n =  \frac 1 {\big((n - \tfrac 1 2)\pi\big)^{2} } \quad\text{and}\quad \varphi_n(t) = \sqrt 2 \sin \big((n -\tfrac 1 2)\pi t\big)
\end{equation}
and the Brownian bridge with covariance function $\widetilde{K}(s,t)=s\wedge t -st$:
\begin{equation}
\label{eigBb}
\widetilde\lambda_n =  \frac 1{\big( \pi  n\big)^{2}}   \quad\text{and}\quad \widetilde\varphi_n(t) = \sqrt 2 \sin \big(\pi  n t\big).
\end{equation}
These formulas are obtained by reduction of the eigenproblem for integral operators to  explicitly solvable boundary 
value problems for ordinary differential equations. Similar kind of reduction also works for a number of 
processes, related to the Brownian motion, for which \eqref{eig} can be put into the framework of the Sturm-Liouville 
type theory \cite{NN04ptrf}.

In this paper we consider the eigenproblem for Gaussian bridges. For a {\em base}  process $X=(X_t, t\in [0,1])$,
the corresponding bridge $\widetilde X=(\widetilde X_t, t\in [0,1])$ is obtained by ``restricting" the trajectories to 
start and terminate at the given points. For Gaussian processes, such restriction amounts to the usual conditioning. 
Hence if $X$ is a centered Gaussian {\em base} process with the starting point $X_0=0$ and covariance function 
$K(s,t)$, the corresponding zero-to-zero bridge  
% with zero terminal point 
is the centered Gaussian process 
$$
\widetilde {X}_t =  X_t - \frac{K(t,1)}{K(1,1)}X_1, \quad t\in [0,1]
$$
with the covariance function 
\begin{equation}\label{barK}
\widetilde{K}(s,t) = K(s,t) - \frac{K(s,1)K(t,1)}{K(1,1)}. 
\end{equation}
Various aspects of general Gaussian bridges are discussed in \cite{GSV07}, \cite{SY14}. 
Besides mathematical interest, they are important ingredients in applications, 
such as statistical hypothesis testing \cite{Lehmann}, exact sampling of diffusions \cite{BS05}, etc.  

\medskip

The covariance operator of the bridge with kernel \eqref{barK} is a rank one perturbation of the 
covariance operator of its base process.  This explains similarity between \eqref{eigBm} and \eqref{eigBb} and 
suggests that the spectra of the two processes must be closely related in general. 
This is indeed the case and one can find an exact expression for the Fredholm determinant of $\widetilde K$ 
in terms of the Fredholm determinant of $K$ even for more general finite rank perturbations (see, e.g.,  \cite{Su72}, Ch.II, 4.6 in 
\cite{KK58}). 
As mentioned above, the precise formulas for the eigenvalues and eigenfunctions of $K$ are rarely known; 
however, the exact asymptotic approximation can be more tractable. 
This raises the following question:

\medskip
\begin{equation}\label{Q} \tag{Q}
\text{ \begin{minipage}{0.8\textwidth}
\em
Can the exact asymptotics of the eigenvalues 
and the eigenfunctions for the bridge be deduced from those of the base process ? 
\end{minipage}
}
\end{equation}

\medskip

A rough answer to this question is given by the general perturbation theory \cite{Kato}, which implies 
that the eigenvalues of $K$ and $\widetilde K$ agree in the leading asymptotic term, as it happens for \eqref{eigBm} and \eqref{eigBb}
(see, e.g., the proof of Lemma 2 in \cite{Br03a}).  
More delicate spectral discrepancies are harder to exhibit and seem to be highly sensitive to the perturbation structure. 
This is vividly demonstrated in the paper \cite{N09b}, where the kernels of the following form  are considered, cf. \eqref{barK}:
\begin{equation}\label{KQ}
\widetilde {K}_Q(s,t) = K(s,t) + Q \psi(s)\psi(t).
\end{equation} 
Here $Q$ is a scalar real valued parameter and $\psi$ is a function in the range of $K$. It turns out that 
for any $Q$ greater than a certain critical value $Q^*$, the spectrum of $\widetilde K_Q$ coincides with that of $K$ 
in the first two asymptotic terms. For $Q:=Q^*$ the spectra depart in the second term. The deviation is quantified in 
\cite{N09b}, when $\psi$ is an image of an $L^2([0,1])$ function, under the action of $K$.
The bridge process under consideration corresponds precisely to the critical case, but with 
$\psi(x)=K(1,x)$ being an image of the distribution $\delta(t-1)$, rather than of a square integrable function; 
hence the approach of \cite{N09b} is not directly applicable here. 

In this paper we will take a different route towards answering the above question, using the particular 
structure of the perturbation inherent to bridges.  
Observe that the eigenproblem $\widetilde{K}  \widetilde\varphi =   \widetilde\lambda  \widetilde\varphi$  
can be written in terms of the covariance operator of the base process  
\begin{equation}\label{beig}
\int_0^1 K(s,t)\widetilde\varphi(s)ds - K(1,t) \int_0^1 K(1,s)\widetilde\varphi(s)ds = \widetilde\lambda \widetilde\varphi(t), \quad t\in [0,1],
\end{equation}
where, without loss of generality, $X$ is assumed to be normalized so that $K(1,1)=1$.  
Taking scalar product with the eigenfunction $\varphi_n$ of $K$ gives
\begin{equation}
\label{phisp}
\langle \widetilde\varphi, \varphi_n\rangle = c  \frac{\lambda_n }{\lambda_n  -\widetilde\lambda}\varphi_n(1),\qquad \widetilde \lambda\not \in  \{\lambda_1,\lambda_2, ...\}
\end{equation}
where  $\displaystyle c:= \int_0^1 K(1,s)\widetilde\varphi(s)ds$. If $K(s,t)$ is continuous, by Mercer's theorem 
$$
K(s,t) = \sum_{n\ge 1}\lambda_n \varphi_n(t)\varphi_n(s) 
$$
where the convergence is absolute and uniform, and hence, in view of \eqref{phisp},
$$
c=\int_0^1 K(1,s)\widetilde\varphi(s)ds =  \sum_{n\ge 1}\lambda_n \varphi_n(1)\langle \widetilde\varphi, \varphi_n  \rangle 
=c\sum_{n\ge 1}  \frac{\lambda_n^2 }{\lambda_n  -\widetilde\lambda}\varphi_n(1)^2
$$
Since $\sum_{n\ge 1}\lambda_n \varphi_n(1)^2=K(1,1)=1$ and $c\ne 0$ whenever $\widetilde \lambda\not\in \{\lambda_1, \lambda_2, ...\}$, 
we obtain the following transcendental equation for the eigenvalues of the bridge
\begin{equation}\label{beq} 
\sum_{k=1}^\infty \frac{\lambda_n }{\lambda_n  -\widetilde\lambda}\varphi_n(1)^2  =0,
\end{equation}
and the corresponding expression for its eigenfunctions:
\begin{equation}
\label{phit}
\widetilde \varphi(t) = c \sum_{n\ge 1}   \frac{\lambda_n }{\lambda_n  -\widetilde\lambda}\varphi_n(1)\varphi_n(t), \quad t\in [0,1].
\end{equation}
Note that the roots of \eqref{beq} are not determined solely by the eigenvalues of the base process, but also require some information 
on its eigenfunctions.

The objective of this paper is to show how equations  \eqref{beq} and \eqref{phit} can be used to construct  
asymptotic approximation for the solutions to the bridge eigenproblem \eqref{beig}, given the exact asymptotics of the 
eigenvalues and eigenfunctions of the corresponding base process.

\section{The main result}

For definiteness we will work with a particular process, though the same approach applies whenever 
similar spectral approximation  for the base process is available
(as, e.g., for the processes considered in \cite{ChKM1}).  
Our study case will be the fractional Brownian motion (fBm), that is, the centered Gaussian process 
$B^H=(B^H_t, t\in [0,1])$ with covariance  function 
\begin{equation}\label{Kst}
K(s,t) = \tfrac 12\left(t^{2H}+s^{2H}-|t-s|^{2H}\right), \quad s,t\in [0,1],
\end{equation}
where $H\in (0,1]$ is its Hurst exponent. 

This is the only $H$-selfsimilar Gaussian process with stationary increments.
For $H=\frac 1 2$ it coincides with the standard Brownian motion, but otherwise has quite different properties. 
In particular, for $H\ne \frac 1 2$, it is neither a semimartingale nor a Markov process. For $H>\frac 12$ the covariance 
sequence of its increments on integers is not summable. This long range dependence property 
makes the fBm a powerful tool in modelling, see \cite{PT17}. 

The fBm has been extensively studied since its introduction in \cite{MvN68} (see, e.g.,  \cite{EM02}, \cite{M08}, \cite{BLL14}, \cite{PT17})
and there seems to be little hope to obtain exact solutions to \eqref{eig}, see \cite{M82}. Hence efficient approximations 
are of significant interest. A few largest eigenvalues and the corresponding eigenfunctions can be approximated numerically, 
as, e.g., in \cite{VT13}, but  relative accuracy of such approximations quickly deteriorates as the $\lambda_n$'s get smaller and, 
in our experience, the problem becomes computationally intractable already for $n\ge 50$.  

\medskip 

Smaller eigenvalues and the corresponding eigenfunctions can be approximated using the following asymptotics 
(see \cite{Br03a, Br03b}, \cite{NN04tpa}, \cite{LP04} for earlier results):

\begin{thm}[Theorem 2.1 in \cite{ChK}]\label{thm-ChK}
\

\medskip

\noindent
{\bf 1.} For the fractional Brownian motion with $H\in (0,1)$, the ordered sequence of the eigenvalues satisfies  
\begin{equation}
\label{lambda}
\lambda_n =  \frac{\sin (\pi H)\Gamma(2H+1)}{ \nu_n^{2H+1}} \qquad n=1,2,...
\end{equation}
where  
$
\nu_n =  \pi n + \pi\gamma_H+O(n^{-1})
$
as $n\to\infty$ and\footnote{constants $\gamma_H$ and $\eta_H$ appear in \cite{ChK} in a slightly more subtle, but equivalent form  }
\begin{equation}
\label{ellH}
 \gamma_H := -\frac 1 2 - \frac {1}{ 2}  \frac { (H-\frac 1 2)^2}{ H+\frac 1 2}.
\end{equation}

\medskip
\noindent
{\bf 2.} The  corresponding normalized eigenfunctions admit the approximation 
\begin{equation}\label{phin}
\begin{aligned}
\varphi_n(t) = &  \sqrt 2 \sin\big( \nu_{n} t+\pi \eta_H\big) +  \\
& \int_0^{\infty} f_0(u) e^{-  t \nu_n u} du +
(-1)^{n}\int_0^{\infty}    f_1(u) e^{-  (1-t) \nu_nu}du+ n^{-1}  r_n(t),
\end{aligned}
\end{equation}
where  $f_0(\cdot)$ and $f_1(\cdot)$ are explicit functions defined in \cite{ChK}  and 
\begin{equation}\label{etaH}
 \eta_H  :=   \frac {H -\frac 3 2}{ 4}   \frac {   H-\frac 12 }{ H+\frac 1 2}.
\end{equation}
The residual  $r_n(t)$ in \eqref{phin} is bounded uniformly over $n\in \mathbb{N}$ and $t\in [0,1]$ and 
\begin{equation}
\label{bndp}
\varphi_n(1) = (-1)^{n} \sqrt{2H+1} \big(1+O(n^{-1})\big), \quad n\to\infty.
\end{equation}
%and their averages satisfy 
%\begin{equation}
%\label{ave}
%\int_0^1 \varphi_n(x)dx =  -\sqrt{\frac{2H+1}{1+\ell_H^2}}\; \nu_n^{-1}.
%\end{equation}

\end{thm}

\medskip 

In principle, a modification of the spectral approximation technique from \cite{Ukai}, \cite{ChK} is applicable to the fractional 
Brownian bridge, \cite{N19}. The goal of this paper is to suggest a general approach to answering question \eqref{Q},
based on the the equations \eqref{beq} and \eqref{phit}. We will demonstrate how the spectral asymptotics of the bridge 
can be derived from that of the base process, using the case of the fBm  as a particular example. 
Specifically, we will prove the following result:

\begin{thm}
\label{main-thm}\

\medskip
\noindent
{\bf 1.} For the fractional Brownian bridge with $H\in (0,1)$, the ordered sequence of the eigenvalues satisfies  
\begin{equation}
\label{blambda}
\widetilde \lambda_n =  \frac{\sin (\pi H)\Gamma(2H+1)}{ \widetilde \nu_n^{2H+1}}, \qquad n=1,2,...
\end{equation}
where  
$
\widetilde \nu_n =  \pi n + \pi\widetilde \gamma_H+O(n^{-1})
$
as $n\to\infty$ and 
$$
  \widetilde \gamma_H :=   \gamma_H +    \frac {H }{H+\frac 1 2}.
$$

\medskip
\noindent
{\bf 2.} The corresponding eigenfunctions admit the approximation 
%\begin{equation}
%\label{tildephin}
%\begin{aligned}
% \widetilde\varphi_n(t) & =  
%\sqrt{2}    \sin\big(\widetilde \nu_n t+\pi \eta_H\big)+ 
%\int_0^{\infty} f_0(u)  e^{-\widetilde \nu_n   t u} du + \\
%&
%(-1)^n \int_{0}^{\infty} e^{-   \widetilde \nu_n (1-t)u } 
%\Big(
% \sin (\pi (\widetilde \gamma_H-\gamma_H))\widetilde f_1(u)du +
%   \cos \pi (\widetilde \gamma_H-\gamma_H) f_1(u) 
%\Big) du + \\
%&
%(-1)^n\sin  \pi (\widetilde \gamma_H-\gamma_H)\int_0^{\infty}   
%\widetilde g_1\big(\widetilde \nu_n (1-t) u\big) f_1(u)  du 
%+
%n^{-1}\log n \,\widetilde r_n(t)
%\end{aligned}
%\end{equation}
\begin{multline}
\label{tildephin}
 \widetilde\varphi_n(t)  =  
\sqrt{2}    \sin\big(\widetilde \nu_n t+\pi \eta_H\big)+ 
\int_0^{\infty} f_0(u)  e^{-\widetilde \nu_n   t u} du + \\
(-1)^n \cos \pi (\widetilde \gamma_H-\gamma_H) \int_{0}^{\infty} f_1(u)  e^{-   \widetilde \nu_n (1-t)u } 
 du + n^{-1}\log n \,\widetilde r_n(t),
\end{multline}
where the residual $\widetilde r_n(t)$ is bounded, uniformly over $n\in \mathbb{N}$ and $t\in [0,1]$. 
\end{thm}

\bigskip 

\begin{rem}\

\medskip 
\noindent
{\bf a.} The eigenvalues of the fBm and its bridge differ by a constant shift 
in the second order asymptotic term of the corresponding ``frequencies'',
$$
 \widetilde\gamma_H-  \gamma_H    =   \frac {H }{H+\frac 1 2}.
$$
It reduces to the familiar constant $\frac 1 2$ in the standard Brownian case $H=\frac 1 2$, cf. \eqref{eigBm} and \eqref{eigBb}. 
The residuals in \eqref{phin} and \eqref{tildephin} differ by the $\log n$ factor, which 
may well be an artifact of the approach. 

\medskip 
\noindent
{\bf b.}
The eigenfunctions of the bridge inherit the oscillatory term in \eqref{tildephin} from the corresponding term of the 
base process \eqref{phin}, however, with a frequency shift. Another modification occurs in the integral terms, 
which are responsible for the boundary layer: their contribution is asymptotically negligible away from the endpoints
of the interval, but is persistent near the boundary.  For the base process, these terms force the eigenfunctions to vanish at 
$t=0$ and approach the alternating values \eqref{bndp} at $t=1$; for the bridge, they push the eigenfunctions to zero at both endpoints. 
Tracing back the definitions of all the functions involved,  it can also be seen that the boundary layer vanishes for $H=\frac 1 2$ and 
the leading asymptotic term in \eqref{tildephin} reduces to the familiar formula \eqref{eigBb} for the standard Brownian bridge.

\medskip 
\noindent 
{\bf c.} 
An additional insight into the problem can be gained by considering a slightly more general perturbation of the base covariance operator, 
cf. \eqref{barK} and \eqref{KQ},
\begin{equation}\label{Kgamma}
\widetilde K(s,t) = K(s,t) + Q \frac{K(s,1)K(t,1)}{K(1,1)}
\end{equation}
with $Q\ge -1$, which corresponds to the bridge for $Q=-1$. 
It can be seen that for $Q>-1$, the method of \cite{ChK} still applies and, with the fBm as the base process, 
shows that the eigenvalues of the perturbed operator coincide with those of the base operator, at least up to the second order.
Further asymptotic terms are negligible, but not uniformly over $Q$, at $Q=-1$ the residual explodes and the second order asymptotics 
changes completely. Hence the case of the bridge is ``critical'', just as it is in the method of  A.Nazarov in \cite{N09b}. 
The alternative technique in this paper applies to the perturbed operator in \eqref{Kgamma}, without breaking  down at $Q=-1$.
Time reversibility of the fractional Brownian bridge, also does not play any role in our approach, at least, explicitly. 

\medskip 
\noindent
{\bf d.} The basic equation \eqref{beq}, which relates the eigenvalues of the bridge to those of the base process, 
involves some information on the eigenfunctions, namely their values at the endpoint of the interval $\varphi_n(1)$. The analogous equation, 
relating the eigenvalues of the base process to those of the bridge has the form
$$
\sum_{n\ge 1} \frac{\widetilde \beta_n^2}{\lambda-\widetilde \lambda_n}=1, \quad \lambda\in\Real_+
$$
where 
$$
\widetilde \beta_n = \int_0^1 K(1,s) \widetilde \varphi_n(s)ds.
$$
Hence by calculations, similar to those in this paper, the exact asymptotics of $\lambda_n$ can be derived, if the sufficiently 
precise asymptotics of both $\widetilde \lambda_n$ and $\widetilde \beta_n$ is known. 

\end{rem}

\section{Proof of Theorem \ref{main-thm}}

\subsection{A preview}

Before giving the full proof, let us consider the special case $H=\frac 1 2$, 
corresponding to the standard Brownian motion. Let us see how the formulas \eqref{eigBb} can be derived from 
\eqref{eigBm}, using the equations \eqref{beq} and \eqref{phit}. 
To this end, it will be convenient to change the variables to 
$$
\mu_k := \frac 1 {\pi\sqrt{\lambda_k}} = k-\tfrac 1 2\quad \text{and}\quad \widetilde \mu = \frac 1{\pi\sqrt{\widetilde \lambda}},
$$
so that in view of \eqref{eigBm}, equation \eqref{beq} becomes
\begin{equation}\label{eqexp} 
g(\widetilde \mu):= \sum_{k=1}^\infty \frac{1 }{(k-\frac 1 2)^2 -\widetilde \mu^2}   =0.
\end{equation}
The explicit formula for this series is well known,
\begin{equation}\label{tgclassic}
g(\widetilde \mu)= \dfrac {\pi }{2\widetilde \mu}\tan  (\pi\widetilde \mu)
\end{equation}
and can be derived by means of the residue calculus. 
It will be instructive to recall the calculation: define the function   
$$
f(z) = \frac {\ctg\big(  \pi(z+\tfrac 1 2)\big) }{z^2-\widetilde \mu^2}, \quad z\in \mathbb{C},
$$
which is meromorphic with the simple poles at $z_\pm =\pm \widetilde \mu$
and $z_k = k -\frac 1 2$. Integrating $f(z)$ over a circular contour of radius $R$ and taking the limit $R\to\infty$,
we obtain 
$$
\Res\{f;z_+\}+\Res\{f;z_-\}+ \sum_{k\in \mathbb{Z}}\Res\{f; z_k\}=0,
$$
by Jordan's lemma. Here the residues are 
\begin{align*}
& \Res\{f;z_+\} = \Res\{f;z_-\}= -\frac {1 }{2\widetilde \mu}\tan  (\pi\widetilde \mu) \\
& \Res\{f; z_k\} = \frac 1 \pi\frac {1 }{(k-\frac 1 2)^2-\widetilde \mu^2}.
\end{align*}
Since the sequence $(k-\tfrac 1 2)^2-\widetilde \mu^2$, $k\in \mathbb{Z}$ is symmetric around $\frac 1 2$, the expression  
\eqref{tgclassic} is obtained and the equation \eqref{eqexp} produces the roots $\widetilde \mu_n = \pi n$, $n=1,2,...$, 
confirming the formula for the eigenvalues in \eqref{eigBb}.

\medskip

The corresponding eigenfunctions can be found using \eqref{phit}:
$$
\widetilde\varphi_n(t) = 2c_n \widetilde \mu^2_n\sum_{k=1}^\infty \frac{(-1)^k \sin \mu_k\pi t}{\widetilde \mu^2_n  -\mu_k^2 }=
2c_n  n^2\sum_{k=1}^\infty \frac{(-1)^k \sin (k-\frac 1 2)\pi t}{n^2  -(k-\frac 1 2)^2 }.
$$
Using similar residue calculus, the series can be computed exactly:
$$
\widetilde\varphi_n(t)=-c_n n \pi (-1)^n\sin \pi nt,
$$
which agrees with the formula in \eqref{eigBb}, after normalizing to the unit norm.

\medskip

The more general case $H\in (0,1)$ is different in two aspects:
\medskip

\begin{enumerate}
\addtolength{\itemsep}{0.7\baselineskip}

\item The function $g(\widetilde \mu)$ for $H\ne \frac 1 2$ involves a power function with non-integer  
exponent (see \eqref{Geq} below) and hence, in addition to the poles, has a discontinuity across the branch cut.  
Consequently the Cauchy theorem cannot be applied as before and a different contour is to be chosen. 
A natural choice is the boundary of the half disk, which lies in the right half plane, but such integration 
produces an additional integral term along the imaginary axis. Asymptotic analysis shows that its contribution is 
non-negligible on the relevant scale for all values of $H$ but $\frac1 2$; thus it is ``invisible'' in the case of 
standard Brownian motion.

\item The exact formulas for the eigenvalues and eigenfunctions for $H\ne \frac 1 2$ are unavailable beyond their precise
asymptotics as in \eqref{lambda}-\eqref{ellH}. It is then reasonable to consider first the perturbed version of the equation 
\eqref{beq}, in which  $\lambda_k$ and $\varphi_k$ are replaced with the corresponding asymptotic approximations from Theorem \ref{thm-ChK}.
This gives the main terms in the eigenvalues formula \eqref{blambda}. It remains then to show that the roots of the perturbed and 
the exact equations get close asymptotically on the suitable scale.  Once the asymptotics of $\widetilde \lambda_n$
becomes available, it can be plugged into \eqref{phit}, along with the expressions for $\lambda_k$ and $\varphi_k(t)$,
to construct the approximations for the bridge eigenfunctions. 

\end{enumerate}

\subsection{The eigenvalues}
Let us change the variable to $\widetilde \mu$ defined by the relation  
$$
\widetilde\lambda =  \frac{\sin (\pi H)\Gamma(2H+1)}{ (\pi\widetilde \mu)^{2H+1}},
$$
in which case equation \eqref{beq} becomes
\begin{equation}\label{Geq}
g(\widetilde \mu):=\sum_{k=1}^\infty \frac{ \varphi_k(1)^2 }{\mu_k^{2H+1}-\widetilde \mu^{2H+1}}  =0,
\end{equation}
where we defined $\mu_k := \nu_k/\pi$. 
Observe that $g(\cdot)$ is continuous, increases on $\Real_+\setminus \{ \mu_k, k\in \mathbb{N}\}$ and 
$$
\lim_{\widetilde\mu\searrow \mu_k }g(\widetilde\mu)=-\infty\quad \text{and}\quad 
\lim_{\widetilde\mu\nearrow \mu_k  }g(\widetilde\mu)=+\infty,\quad k\in \mathbb{N}.
$$ 
Consequently it has the unique root $\widetilde \mu_n$ at each one of the intervals $\big(\mu_n,\mu_{n+1}\big)$.

\medskip 

In view of the asymptotics \eqref{lambda}-\eqref{phin}, it makes sense to consider first the ``approximate'' perturbed equation
\begin{equation}\label{tildeGeq}
g^a(\widetilde \mu):=\sum_{k=1}^\infty \frac{ 2H+1 }{(k + \gamma_H)^{2H+1}-\widetilde\mu^{2H+1}}  =0,
\end{equation}
where the eigenvalues and eigenfunctions of the base process are replaced with their asymptotic approximations. 
The next step is to argue that the roots of the exact and perturbed equations \eqref{Geq} and \eqref{tildeGeq} are
close on an appropriate scale, asymptotically as $n\to\infty$.
These steps are implemented in Lemmas \ref{lem1} and \ref{lem2} respectively, which together imply assertion 1 of Theorem \ref{main-thm}.

\begin{lem}\label{lem1} The unique root $\widetilde \mu^a_n\in (n+\gamma_H,n+1+\gamma_H)$ of equation \eqref{tildeGeq} satisfies 
\begin{equation}\label{tildemua}
  \widetilde \mu_n^a =  n   +\gamma_H + \frac {H}{H+\frac 1 2}  + O(n^{-1}), \quad n\to\infty.
\end{equation}
\end{lem}

\begin{proof}

A useful expression can be found for $ g^a(\widetilde\mu)$ using the residue calculus. To this end
note that the principal branch of the function 
$$
f(z) := \frac{\ctg \big(\pi (z-\gamma_H)\big)}{z^{2H+1}-\widetilde\mu^{2H+1}}, \quad z\in \mathbb{C}\setminus \Real_-
$$ 
is meromorphic on the right half plane with simple poles at $z_0:=\widetilde\mu$ and $z_k := k+\gamma_H$, $k=1,2,...$
Note that $z_k >0$ since $\gamma_H\in (-\frac 3 4, -\frac 1 2)$  for $H\in (0,1)$. 
Integrating this function over the boundary of the half disk of radius $R\in \mathbb{N}$
in the right half plane gives
$$
\int_{iR}^{-iR} f(z)dz + \int_{C_R} f(z)dz =2\pi i \Res\big\{f,z_0\big\} + 2\pi i\sum_{k=1}^{R} \Res\big\{f, z_k\big\},
$$ 
where $C_R$ denotes the semicircular arc. Since $\ctg(\cdot)$ is bounded on $C_R$, by Jordan's lemma 
the integral over $C_R$ vanishes as $R\to\infty$ and we obtain 
\begin{equation}
\label{fla}
\frac 1 {2\pi } \int_{-\infty}^{\infty} f(it)d t  =     -  \Res\big\{f,z_0\big\} -  \sum_{k=1}^{\infty} \Res\big\{f, z_k\big\}.
\end{equation}
Computing the residues  
\begin{align*}
& 
\Res\big\{f,z_0\big\}=
%\lim_{z\to \widetilde\mu} (z-\widetilde\mu) f(z)= 
\ctg (\pi (\widetilde\mu-\gamma_H))\lim_{z\to \widetilde\mu}  \frac{z-\widetilde\mu}{z^{2H+1}-\widetilde\mu^{2H+1}} = 
\ctg (\pi (\widetilde\mu-\gamma_H))\frac {\widetilde\mu^{-2H}}{2H+1}
\\
&
\Res\big\{f,z_k\big\}= %\lim_{z\to k+\gamma_H} (z-k-\gamma_H)\frac{1}{z^{2H+1}-\widetilde\mu^{2H+1}} \ctg (\pi (z-\gamma_H))=
\frac{1}{(k+\gamma_H)^{2H+1}-\widetilde\mu^{2H+1}} \lim_{z\to 0} z \frac{\cos (\pi z)}{\sin (\pi z)}=
\frac 1 \pi\frac{1}{(k+\gamma_H)^{2H+1}-\widetilde\mu^{2H+1}}
\end{align*}
and plugging these expressions into \eqref{fla} gives the formula 
\begin{align*}
g^a(\widetilde\mu)
%=\sum_{k=1}^{\infty} \frac{2H+1}{(k+\gamma_H)^{2H+1}-\widetilde\mu^{2H+1}}
= & -\frac {2H+1} {2  } \int_{-\infty}^{\infty} f(it)d t      - \frac {\pi}{ \widetilde\mu^{2H}} \ctg (\pi (\widetilde\mu-\gamma_H))= \\
& 
=-\frac {2H+1}   {\widetilde\mu^{2H }}
\Re\left\{
\int_{0}^{\infty} 
\frac{\ctg (\pi (i\tau \widetilde\mu -\gamma_H))}{(i\tau)^{2H+1}-1}
d \tau 
\right\}
 - \frac {\pi}{ \widetilde\mu^{2H}} \ctg (\pi (\widetilde\mu-\gamma_H)).
\end{align*}
Hence the equation \eqref{tildeGeq} takes the form
\begin{equation}
\label{ctgeq}
\ctg (\pi (\widetilde\mu-\gamma_H)) 
%=
%-\frac {2H+1} {2\pi } \widetilde\mu^{2H}\int_{-\infty}^{\infty} \frac{\ctg (\pi (it-\gamma_H))}{(it)^{2H+1}-\widetilde\mu^{2H+1}}d t
=
-\frac {2H+1} { \pi }   \Re
\left\{\int_{0}^{\infty} \frac{\ctg (\pi (i\tau\widetilde\mu-\gamma_H))}{(i\tau)^{2H+1}-1}d \tau \right\}.
\end{equation}
Let $\widetilde\mu_n^a$  be the unique root of \eqref{tildeGeq} in the interval $(n+\gamma_H,n+1+\gamma_H)$, then 
\begin{equation}
\label{eqeq}
\begin{aligned}
\int_{0}^{\infty} \frac{\ctg (\pi (i\tau \widetilde \mu_n^a -\gamma_H))}{(i\tau )^{2H+1}-1} d \tau & =
\frac 1 i\int_{0}^{\infty} 
\frac
{
1+ e^{ -2\pi  (\tau \widetilde \mu_n^a +   i\gamma_H) } 
}
{
1-e^{ -2\pi  (\tau \widetilde\mu_n^a +  i\gamma_H) } 
}
\frac{1}{(i\tau )^{2H+1}-1} d \tau \\
&
=
\frac 1 i\int_{0}^{\infty} 
\frac{1}{(i\tau )^{2H+1}-1} d \tau + R_n.
\end{aligned}
\end{equation}
The residual $R_n$ here satisfies
\begin{equation}\label{Rnest}
\begin{aligned}
\widetilde \mu^a_n\big|R_n\big| \le & 
%\int_{0}^{\infty} 
%\frac
%{
%2 e^{ -2\pi  \tau \widetilde \mu_n^a } 
%}
%{
%\big|
%1-e^{ -2\pi  (\tau \widetilde\mu_n^a +  i\gamma_H)} 
%\big|
%}
%\frac{1}{\big|(i\tau )^{2H+1}-1\big|} d (\tau \widetilde \mu^a_n) \le \\
%&
%=  \int_{0}^{\infty} 
%\frac
%{
%2 e^{ -2\pi  t } 
%}
%{
%\big|
%1-e^{ -2\pi  t - 2\pi i\gamma_H} 
%\big|
%}
%\frac{1}{\big|(it/\widetilde \mu^a_n )^{2H+1}-1\big|} d t
%
\int_{0}^{\infty} 
\left|
\frac
{
2 e^{ -2\pi  t } 
}
{
1-e^{ -2\pi  t - 2\pi i\gamma_H} 
}
\frac{1}{(it/\widetilde \mu^a_n )^{2H+1}-1 } 
\right| d t \le \\
&
\int_{0}^{\infty} 
\frac{2 e^{ -2\pi  t }}{\big|(it/\widetilde \mu^a_n )^{2H+1}-1\big|} d t
\xrightarrow[n\to\infty]{} \frac 1 \pi,
\end{aligned}
\end{equation}
where the second inequality holds since $\cos(2\pi \gamma_H)\le 0$.
The real part of the integral on the right hand side of \eqref{eqeq} can be computed explicitly:
\begin{align*}
&
\frac 1 i\int_{0}^{\infty} \frac{1}{(i\tau )^{2H+1}-1} d \tau =   
\frac 1 {i^{2H+2}}\int_{0}^{\infty} \frac{1}{ \tau^{2H+1}-1/i^{2H+1}} d \tau =\\
&
\frac 1 {i^{2H+2}}\frac 1{2H+1}  \int_{0}^{\infty} \frac{u^{ \frac {1}{2H+1}-1}}{ u-1/i^{2H+1}} d u=
 \frac { \pi}{2H+1} \frac { e^{ -\frac {1}{2H+1}\pi i} }{\sin \frac {\pi}{2H+1} }  
\end{align*}
and hence 
\begin{equation}\label{intexp}
\Re\left\{\frac 1 i\int_{0}^{\infty} 
\frac{1}{(i\tau )^{2H+1}-1} d \tau\right\}= \frac { \pi}{2H+1}  \ctg \frac {\pi}{2H+1}.
\end{equation}

Plugging \eqref{Rnest} and \eqref{intexp} into \eqref{eqeq} and \eqref{ctgeq} and recalling that 
$\widetilde\mu_n^a-\gamma_H\in (n,n+1)$, we obtain the asymptotics claimed in \eqref{tildemua}.
\end{proof}

The following refinement of  \eqref{bndp} allows to obtain slightly more accurate estimate for the residual in 
the eigenvalues approximation.

\begin{lem}
There exists a finite integer $n_0$, such that 
\begin{equation}
\label{bndp_exact}
\varphi_n(1) = (-1)^{n} \sqrt{2H+1}, \quad \forall n\ge n_0.
\end{equation}
\end{lem}

\begin{proof}
The $m$-th iterate of the kernel satisfies 
$$
K^{(m)}(1,1) = \sum_{n=1}^\infty \lambda_n^m \varphi^2_n(1)\quad \text{and}\quad
\int_0^1 K^{(m)}(s,s)ds = \sum_{n=1}^\infty \lambda_n^m, \quad m\in \mathbb{N}.
$$
A direct calculation shows that for the kernel \eqref{Kst}, which is self-similar in the sense  
$$
K(cs,ct)=c^{2H} K(s,t)\quad \forall c\in \Real_+,
$$
the following identity holds
$$
\int_0^1 K^{(m)}(s,s)ds = \frac 1 {2H+1} K^{(m)}(1,1).
$$
Therefore for all $m\in \mathbb{N}$, 
\begin{equation}\label{series}
\sum_{n=1}^\infty \lambda_n^m \big(\varphi^2_n(1)-(2H+1)\big) =0.
\end{equation}
Multiplying by $\lambda^m$ and summing up over $m$ we obtain 
$$
\sum_{n=1}^\infty a_n \frac{\lambda \lambda_n}{1-\lambda \lambda_n}  =0,\quad \lambda\in (0,\lambda_1^{-1}),
$$
where  $a_n:= \varphi^2_n(1)-(2H+1)$. Let $k_1\ge 1$ be the multiplicity of $\lambda_1$, then  
$$
\frac{\lambda \lambda_1}{1-\lambda \lambda_1} \sum_{n=1}^{k_1} a_n + \sum_{n=k_1+1}^\infty a_n \frac{\lambda \lambda_n}{1-\lambda \lambda_n} =0.
$$
The second term on the left hand side converges to a finite limit as $\lambda\nearrow \lambda_1^{-1}$, hence we 
must have $\sum_{n=1}^{k_1} a_n=0$.
Therefore, the first $k_1$ terms in series \eqref{series} vanish and we can repeat the argument for $\lambda_2$  and so on.
Asymptotics \eqref{lambda} implies that $\lambda_n$'s have unit multiplicity for all $n$ large enough and hence 
$a_n=\varphi_n^2(1)-(2H+1)=0$ for all such $n$'s. Equality \eqref{bndp_exact} follows from \eqref{bndp}.

\end{proof}

The next step is to show that the roots of \eqref{Geq} and \eqref{tildeGeq} are close on a suitable scale:

\begin{lem}\label{lem2}
Let $\widetilde \mu_n$ and $\widetilde \mu^a_n$ be the unique roots of equations \eqref{Geq} and \eqref{tildeGeq} in the interval 
$(\mu_n,\mu_{n+1})$. 
Then 
\begin{equation}\label{apprmu}
\widetilde \mu_n- \widetilde \mu^a_n = O(n^{-1}), \quad n\to\infty.
\end{equation}
\end{lem}

\begin{proof}
 
Suppose $f:I \mapsto \Real$ is a function on an open interval $I$ with $\frac {d}{dx}f(x)\ge r>0$ and 
a root $x_0\in I$. Let $h:I\mapsto \Real$ be a continuous strictly increasing function with $\sup_{x\in I}|f(x)-h(x)|\le b$ 
and assume that $[x_0-b/r, x_0+b/r]\subset I$. Then $h$ has a unique root $y_0$ and it satisfies 
\begin{equation}\label{bndosc}
|y_0-x_0|\le b/r.
\end{equation}

We will apply this elementary bound to $f:=g^a$ and $h:= g$ on the interval $I_n$ with the endpoints at $n+\widetilde \gamma_H \pm \delta$, 
where $\delta>0$ is fixed and chosen to be small enough so that $I_n\subset (\mu_n,\mu_{n+1})$. 
Recall that by Lemma \ref{lem1}, the unique root $\widetilde \mu^a_n\in (n+\gamma_H,n+1+\gamma_H)$ of $g^a$ 
belongs to $I_n$ for all sufficiently large $n$. The function $g^a$ is differentiable on $\Real_+ \setminus \{k+\gamma_H: k\in \mathbb{N}\}$
and 
\begin{equation}\label{qtagbnd}
\begin{aligned}
\inf_{\widetilde\mu\in I_n}\frac d {d\widetilde\mu}g^{a}(\widetilde\mu) = &
\inf_{\widetilde\mu\in I_n}\sum_{k=1}^\infty \frac{ (2H+1)^2\widetilde\mu^{2H} }{\big((k + \gamma_H)^{2H+1}-\widetilde\mu^{2H+1}\big)^2} \ge \\
&
\inf_{\widetilde\mu\in I_n}\frac{ \widetilde\mu^{2H} }{\big((n + \gamma_H)^{2H+1}-\widetilde\mu^{2H+1}\big)^2} \ge \\
&
\frac{ (n+\widetilde \gamma_H-\delta)^{2H} }{\big((n + \gamma_H)^{2H+1}-(n+\widetilde \gamma_H+\delta)^{2H+1}\big)^2} > c n^{-2H}
\end{aligned}
\end{equation}
with a constant $c>0$. 

Next let us estimate the oscillation of $g^a(\widetilde\mu)-g(\widetilde\mu)$ on $I_n$. In view of \eqref{bndp_exact},
\begin{align*}
&
g^a(\widetilde\mu)-g(\widetilde\mu) =  
\sum_{k=1}^{n_0} \frac{ (2H+1)-\varphi_k(1)^2 }{(k + \gamma_H)^{2H+1}-\widetilde\mu^{2H+1}}
+ \\
&
\sum_{k=1}^\infty \varphi_k(1)^2\frac{ \mu_k^{2H+1} -(k + \gamma_H)^{2H+1} }
{
\big(
(k + \gamma_H)^{2H+1}-\widetilde\mu^{2H+1}
\big)
\big(
\mu_k^{2H+1}-\widetilde\mu^{2H+1}
\big)} =: D_1(\widetilde\mu) + D_2(\widetilde\mu).
\end{align*}
Since the summation in the first term is finite, 
$$
\sup_{\widetilde \mu \in I_n}\big|D_1(\widetilde\mu)\big| = O(n^{-2H-1}).
$$
Due to the asymptotics $\mu_k  = k + \gamma_H+ O(k^{-1})$, the second term  satisfies 
$$
\sup_{\widetilde \mu\in I_n}\big|D_2(\widetilde\mu)\big|\lesssim  \sup_{\widetilde \mu\in I_n}S_{1,\infty}(\widetilde \mu),
$$
where we defined  
$$
S_{i,j}(\widetilde \mu) := 
\sum_{k=i}^j \frac{ k^{2H-1} }
{
\big|\big(
(k + \gamma_H)^{2H+1}-\widetilde\mu^{2H+1}
\big)
\big(
\mu_k^{2H+1}-\widetilde\mu^{2H+1}
\big)\big|}.
$$
Hereafter $x\lesssim y$ stands for inequality $x\le C y$ with a constant $C$, whose value is of no importance. 
Let us split the sum $S_{1,\infty}(\widetilde \mu)$ into three parts and estimate each one separately. 
Obviously, $\sup_{\widetilde \mu\in I_n} S_{n,n}(\widetilde \mu) = O(n^{-2H-1})$ and  
\begin{align*}
\sup_{\widetilde \mu\in I_n}S_{1,n-1}(\widetilde \mu) \lesssim &
\sum_{k=1}^{n-1} \frac{ k^{2H-1} }
{
 \big(
(k + \gamma_H)^{2H+1}-(n+\widetilde \gamma_H-\delta)^{2H+1}
\big)^2
} \lesssim   \\
&
n^{-2H-2} \int_0^{1-1/n} \frac {x^{2H-1}}{(x^{2H+1}-1)^2}dx = O(n^{-2H-1}), \quad n\to \infty.
\end{align*}
Similar estimate holds for $\sup_{\widetilde \mu\in I_n}S_{n+1,\infty}(\widetilde \mu)$ and therefore 
$$
\sup_{\widetilde \mu\in I_n}\big|D_2(\widetilde\mu)\big| = O(n^{-2H-1}).
$$ 
It follows that
$$
\sup_{\widetilde\mu \in I_n}\big|g^a(\widetilde\mu)-g(\widetilde\mu)\big|= O(n^{-2H-1}).
$$ 
Plugging this estimate and \eqref{qtagbnd} into \eqref{bndosc} gives the claimed asymptotics.
\end{proof}

\begin{rem}
Without refinement \eqref{bndp_exact}, the residual of order $O(n^{-1})$ in \eqref{apprmu} would have been inflated by $\log n$ factor. 
\end{rem}

\subsection{The eigenfunctions} 
Approximation \eqref{tildephin} is obtained by plugging the asymptotic expressions \eqref{lambda}, \eqref{phin} and \eqref{blambda}  
into the formula \eqref{phit}:
\begin{equation}\label{phimu}
\widetilde\varphi_n(t) = c_n \sum_{k=1}^\infty \frac{\lambda_k }{\lambda_k  -\widetilde \lambda_n}\varphi_k(1) \varphi_k(t)=-
c_n\widetilde \mu_n^{2H+1}\sum_{k=1}^\infty \frac{1 }{\mu_k^{2H+1}-\widetilde \mu_n^{2H+1}}\varphi_k(1) \varphi_k(t)
\end{equation}
where we set $\mu_n := \nu_n/\pi$ and $\widetilde \mu_n := \widetilde \nu_n/\pi$ as in Lemma \ref{lem1}.

\medskip

As before, we will first replace the exact values by their leading asymptotic terms and then 
argue that the error, thus introduced, is negligible on the suitable scale. To this end, define, cf. \eqref{phin} and \eqref{bndp_exact},
\begin{align}
\label{phi1}
\widetilde\varphi_n^{1,a}(t) = & -c_n
\widetilde \mu_n^{2H+1}\sqrt{2H+1}
\sqrt 2 \sum_{k=1}^\infty 
\frac
{ (-1)^k  \sin\big( \pi(k+\gamma_H)  t+\pi\eta_H\big)}
{ (k+\gamma_H)^{2H+1}-\widetilde\mu_n^{2H+1}    
}
\\
\label{phi2}
\widetilde\varphi_n^{2,a}(t) = &-c_n
\widetilde \mu_n^{2H+1}\sqrt{2H+1}
\int_0^{\infty} f_0(u) \left(
\sum_{k=1}^\infty 
\frac
{ (-1)^k e^{-   (k+\gamma_H) \pi t u}}
{ (k+\gamma_H)^{2H+1}-\widetilde\mu_n^{2H+1}
}
\right) 
du
\\
\label{phi3}
\widetilde\varphi_n^{3,a}(t) = &-c_n
\widetilde \mu_n^{2H+1}\sqrt{2H+1}
\int_0^{\infty}    f_1(u)
\left(
\sum_{k=1}^\infty 
\frac
{
 e^{-    (k+\gamma_H)\pi(1-t) u}
} 
{ 
(k+\gamma_H)^{2H+1}-\widetilde\mu_n^{2H+1}
}     
 \right)     du
\end{align}
where $\eta_H$ is the constant defined in \eqref{etaH}. 

\medskip

\begin{lem}\label{lem3.3}\
The function $\widetilde\varphi^a_n(t)=\widetilde\varphi_n^{1,a}(t)+\widetilde\varphi_n^{2,a}(t)+\widetilde\varphi_n^{3,a}(t)$,
satisfies  
\begin{equation}\label{general}
\begin{aligned}
&
\widetilde\varphi^a_n(t)  /\|\widetilde\varphi^a_n\| =  
\sqrt{2}    \sin\big(\widetilde \nu_n t+\pi\eta_H\big)+ 
\int_0^{\infty} f_0(u)  e^{-\widetilde \nu_n   t u} du + \\
&
(-1)^n \int_{0}^{\infty} e^{-   \widetilde \nu_n (1-t)u } 
\Big(
 \sin (\pi (\widetilde \gamma_H-\gamma_H))\widetilde f_1(u)  +
   \cos \pi (\widetilde \gamma_H-\gamma_H) f_1(u) 
\Big) du + \\
&
(-1)^n\sin  \pi (\widetilde \gamma_H-\gamma_H)\int_0^{\infty}   
\widetilde g_1\big(\widetilde \nu_n (1-t) u\big) f_1(u)  du 
+
n^{-1} \widetilde r_n(t),
\end{aligned}
\end{equation}
where $\widetilde f_1$ and $\widetilde g_1$ are explicit functions, defined in \eqref{tildef1} and \eqref{psifla} below, and 
the residual $\widetilde r_n(t)$ is bounded uniformly over $n\in \mathbb{N}$ and $t\in [0,1]$. 

\end{lem}

\begin{proof}
The claimed approximation is obtained by finding the leading term asymptotics of the functions in \eqref{phi1}-\eqref{phi3} 
and normalizing their sum by a suitable common factor. 

\medskip

\noindent
{ 1) \em Asymptotics of \eqref{phi1}}.
For a fixed $t\in [0,1]$ and $\widetilde\mu>0$, consider the series  
$$
h(\widetilde\mu):=\sum_{k=1}^\infty 
\frac
{ (-1)^k \sin\big(\pi(k+\gamma_H) t+\pi\eta_H\big)}
{  (k+\gamma_H)^{2H+1} -\widetilde\mu^{2H+1}   
}
,
$$
for which a closed form formula can be found by means of residue calculus as in Lemma \ref{lem1}.
To this end, consider the principal branch of the function 
\begin{equation}\label{ffla}
f(z):=\frac{\sin\big(\pi(z t+\eta_H)\big)}{z^{2H+1}-\widetilde\mu^{2H+1}}\frac {1} {\sin (\pi (z-\gamma_H))},\quad z\in \mathbb{C} \setminus \Real_-,
\end{equation}
which is meromorphic on the right half plane with simple poles at $z_0 := \widetilde\mu$ and $z_k = k + \gamma_H$, $k\in \mathbb{N}$.
Integrating $f(z)$ over the half disc boundary in the right half plane gives
$$
 \int_{iR}^{-iR}f(z)dz+\int_{C_R} f(z)dz = 2\pi i\Res\{f; z_0\} + 2\pi i\sum_{k=1}^R \Res\{f;z_k\},
$$
where $C_R$ stands for the  semicircular arc with radius $R\in \mathbb{N}$. The ratio of sines in \eqref{ffla}
is bounded for any $t\in [0,1]$ and therefore, applying Jordan's lemma, we get
\begin{equation}
\label{res2}
\frac1 {2\pi }\int_{-\infty}^{\infty}f(it)d t =  -\Res\{f; z_0\} -  \sum_{k=1}^\infty \Res\{f;z_k\},
\end{equation}
with the residues 
\begin{align*}
\Res\{f; z_0\} =\ &  %\frac {\sin\big(\pi(\widetilde\mu t+\eta_H)\big)} {\sin (\pi (\widetilde\mu-\gamma_H))} \lim_{z\to \widetilde\mu}  \frac{z-\widetilde\mu}{z^{2H+1}-\widetilde\mu^{2H+1}}=
 \frac {\sin\big(\pi(\widetilde\mu t+\eta_H)\big)} {\sin (\pi (\widetilde\mu-\gamma_H))} 
 \frac 1 {2H+1}\frac 1{\widetilde\mu^{2H}}, \\
\Res\{f; z_k\} =\ &
% \lim_{z\to k+\gamma_H} (z-k-\gamma_H) \frac{1}{z^{2H+1}-\widetilde\mu^{2H+1}}\frac {\sin\big(\pi(z t+\eta_H)\big)} {\sin (\pi (z-\gamma_H))}=
%   \frac{\sin\big(\pi((k+\gamma_H) t+\eta_H)\big)}{(k+\gamma_H)^{2H+1}-\widetilde\mu^{2H+1}} \lim_{z\to k+\gamma_H}
%   \frac {z-k-\gamma_H} {\sin (\pi (z-\gamma_H))} = \\
%   &
\frac{\sin\big(\pi(k+\gamma_H) t+\eta_H)\big)}{(k+\gamma_H)^{2H+1}-\widetilde\mu^{2H+1}} \frac{(-1)^k}{\pi}.
\end{align*}
Plugging these expressions into \eqref{res2}, we get 
\begin{align*}
h(\widetilde\mu) = & - \frac {\pi} {2H+1} \frac 1{\widetilde\mu^{2H}} \frac {\sin\big(\pi(\widetilde\mu t+\eta_H)\big)} {\sin (\pi (\widetilde\mu-\gamma_H))} \\
&
- \frac 1 {\widetilde\mu^{2H}}\Re\left\{
\int_{0}^{\infty} \frac{1}{(iu )^{2H+1}-1}\frac {\sin\big(\pi(iu\widetilde\mu t+\eta_H)\big)} {\sin (\pi (iu\widetilde\mu-\gamma_H))}d u 
\right\}.
\end{align*}
The second term simplifies to  
\begin{align*}
& \Re\left\{\int_{0}^{\infty} \frac{1}{(iu )^{2H+1}-1}\frac {\sin\big(\pi(iu\widetilde\mu t+\eta_H)\big)} {\sin (\pi (iu\widetilde\mu-\gamma_H))}d u\right\} = \\
%&
%\Re\left\{\int_{0}^{\infty} \frac{e^{-\pi  u\widetilde\mu (1-t)-i\pi(\eta_H +\gamma_H) }}{(iu )^{2H+1}-1}
%d u
%\right\}
%+ R(t,\widetilde\mu) = \\
&
\int_{0}^{\infty} e^{-\pi  u\widetilde\mu (1-t) }\Re\left\{ \frac{e^{-i\pi(\eta_H +\gamma_H) }}{(iu )^{2H+1}-1}
\right\}d u + R(t,\widetilde\mu),
\end{align*}
with the residual satisfying  
$$
\widetilde\mu \big|R(t,\widetilde\mu)\big| \le \widetilde\mu \int_{0}^{\infty} \frac{2e^{ -\pi u\widetilde\mu    }}{\big|(iu )^{2H+1}-1\big|}d u=
\int_{0}^{\infty} \frac{2e^{ -\pi s    }}{\big|(is /\widetilde\mu)^{2H+1}-1\big|}d s\xrightarrow[\widetilde\mu\to\infty]{}\frac 2 \pi.
$$
%where the inequality holds since $\cos (2\pi \gamma_H)<0$.

\medskip

Plugging these expressions back gives
\begin{equation}
\begin{aligned}\label{phi1a}
\widetilde\varphi_n^{1,a}(t)\simeq &
\sqrt{2}    \sin\big(\pi(\widetilde \mu_n t+\eta_H)\big)+ \\
& (-1)^n \sin \big(\pi (\widetilde \gamma_H-\gamma_H)\big)
\int_{0}^{\infty}  e^{-  \pi \widetilde \mu_n (1-t)u }
 \widetilde f_1(u)du + n^{-1} \widetilde r_n^{(1)}(t),
\end{aligned}
\end{equation}
where $x\simeq y$ means $x= C y$ with a constant $C$ and we  normalized by the factor
\begin{equation}\label{cn}
\widetilde c_n := c_n \widetilde \mu_n \frac{\pi}{\sqrt{2H+1}} \frac {(-1)^n}{\sin \pi(\widetilde \gamma_H-\gamma_H)}.
\end{equation}
It can be seen (as in the calculation, concluding section 5.1.6. in \cite{ChK}), that the norm of the integral term 
in \eqref{phi1a} is of order $O(n^{-1})$ and hence the norm of $\widetilde\varphi_n^{1,a}$
is asymptotic to 1 as $n\to \infty$. 
The residual $\widetilde r_n^{(1)}(t)$ is uniformly bounded over $n\in \mathbb{N}$ and $t\in [0,1]$
and the function $\widetilde f_1$ is given by the formula
\begin{equation}\label{tildef1}
\widetilde f_1(u):=\frac  {2H+1}{\pi}\sqrt{2}\Re\left\{ \frac{e^{-i\pi(\eta_H +\gamma_H) }}{(iu )^{2H+1}-1}
\right\}.
\end{equation}

\medskip 
\noindent
2) {\em Asymptotics of \eqref{phi2}}.
A closed form expression for the series  
$$
h(\widetilde\mu):=\sum_{k=1}^\infty 
\frac
{ (-1)^k }
{ (k+\gamma_H)^{2H+1}- \widetilde\mu^{2H+1}
}
e^{-   (k+\gamma_H) \pi t u}
$$
can be found by integrating the principal branch of the function 
$$
f(z):= 
\frac
{ e^{-   z \pi t u} }
{ z^{2H+1}- \widetilde\mu^{2H+1}
}
\frac 1{\sin (\pi (z-\gamma_H))}, \qquad z\in \mathbb{C}\setminus \Real_-
$$
over the half disk boundary in the right half plane. As before,
$$
\frac1 {2\pi }\int_{-\infty}^{\infty}f(it)d t =  -\Res\{f; z_0\} -  \sum_{k=1}^\infty \Res\{f;z_k\},
$$
with the same poles as defined above. The residues are given by  
\begin{align*}
& \Res\{f; z_0\} = 
\frac 1 {2H+1} \frac 1{\widetilde\mu^{2H}}
\frac {e^{-   \widetilde\mu \pi t u}}{\sin (\pi (\widetilde\mu-\gamma_H))}
\\
&
\Res\{f; z_k\} = \frac 1 \pi
\frac
{(-1)^k }
{ (k+\gamma_H)^{2H+1}- \widetilde\mu^{2H+1}
}e^{-   (k+\gamma_H) \pi t u} 
\end{align*}
and therefore 
\begin{align*}
h(\widetilde\mu) = & 
-\frac \pi {2H+1} \frac 1{\widetilde\mu^{2H}}\frac {e^{-   \widetilde\mu \pi t u}}{\sin (\pi (\widetilde\mu-\gamma_H))}
\\
&
- \frac 1{\widetilde\mu^{2H}}\Re
\left\{
\int_{0}^{\infty}
\frac
{ e^{-   is\widetilde\mu \pi t u} }
{ (is )^{2H+1}- 1
}
\frac 1{\sin (\pi (is\widetilde\mu-\gamma_H))}
d s 
\right\}.
\end{align*}
The integral term  satisfies 
$$
\left|
\int_{0}^{\infty}
\frac
{ e^{-   is\widetilde\mu \pi t u} }
{ (is )^{2H+1}- 1
}
\frac 1{\sin (\pi (is\widetilde\mu-\gamma_H))}
d s 
\right| \le  \frac 1 {\widetilde\mu}
\int_{0}^{\infty}
\frac
{ 2e^{-\pi   s} }
{\big| (is/\widetilde\mu )^{2H+1}- 1\big|
}
d s 
$$
and hence, normalizing by the constant \eqref{cn}, we get
$$
\widetilde\varphi_n^{2,a}(t) \simeq 
\int_0^{\infty} f_0(u)  e^{-   \widetilde \mu_n \pi t u} du + n^{-1} \widetilde r_n^{(2)}(t)
$$
with a uniformly bounded residual $\widetilde r_n^{(2)}$.

\medskip 
\noindent
3) {\em Asymptotics of \eqref{phi3}}. An explicit formula for the series 
$$
h(\widetilde\mu):= \sum_{k=1}^\infty 
\frac
{
e^{-    (k+\gamma_H)\pi(1-t) u}
}
{ 
(k+\gamma_H)^{2H+1}-\widetilde\mu^{2H+1}
}
$$
is obtained by integrating the principal branch of the function 
$$
f(z):= 
\frac
{e^{-    z\pi(1-t) u}}
{ 
z^{2H+1}-\widetilde\mu^{2H+1}
}
\ctg (\pi (z-\gamma_H)),\qquad z\in \mathbb{C}\setminus \Real_-
$$
over the same contour as above:
$$
\frac 1 {2\pi}\int_{-\infty}^{\infty} f(it )dt    = -   \Res\{f; z_0\}-    \sum_{k=1}^\infty \Res\{f; z_k\}.
$$
The residues are 
\begin{align*}
\Res\{f; z_0\}  &= \frac 1 {2H+1}\frac 1 {\widetilde\mu^{2H}}
e^{-    \widetilde\mu\pi(1-t) u}
\ctg (\pi (\widetilde\mu-\gamma_H))
\\ 
\Res\{f; z_k\}  &=
\frac 1\pi 
\frac
{e^{-    (k+\gamma_H)\pi(1-t) u}}
{ 
(k+\gamma_H)^{2H+1}-\widetilde\mu^{2H+1}
}
\end{align*}
and therefore 
%$$
%h(\widetilde\mu)=  
%-\frac 1 {2 }\int_{-\infty}^{\infty} f(is )ds   
%-   \frac \pi {2H+1}\frac 1 {\widetilde\mu^{2H}}e^{-    \widetilde\mu\pi(1-t) u}\ctg (\pi (\widetilde\mu-\gamma_H)).
%$$
\begin{align*}
h(\widetilde\mu)= &  
-   \frac \pi {2H+1}\frac 1 {\widetilde\mu^{2H}}e^{-    \widetilde\mu\pi(1-t) u}\ctg (\pi (\widetilde\mu-\gamma_H))
 \\
&
-\frac 1 {\widetilde\mu^{2H}}\Im\left\{
 \int_{0}^{\infty}
\frac
{e^{-    is\widetilde\mu\pi(1-t) u}}
{ 
(is)^{2H+1}-1
}
ds
\right\}
-\frac 1 {\widetilde\mu^{2H}} R(\widetilde\mu).
\end{align*}
%The integral term satisfies  
%\begin{equation}\label{intf}
%\begin{aligned}
%\frac 1 {2 }\int_{-\infty}^{\infty} f(is )ds =& \,\Re\left\{
% \int_{0}^{\infty}
%\frac
%{e^{-    is\pi(1-t) u}}
%{ 
%(is)^{2H+1}-\widetilde\mu^{2H+1}
%}
%\ctg (\pi (is-\gamma_H))ds
%\right\} = \\
%&
%-\frac 1 {\widetilde\mu^{2H}}\Re\left\{ i
% \int_{0}^{\infty}
%\frac
%{e^{-    is\widetilde\mu\pi(1-t) u}}
%{ 
%(is)^{2H+1}-1
%}
%\right\}
%ds 
%-\frac 1 {\widetilde\mu^{2H}} R(\widetilde\mu)
%\end{aligned}
%\end{equation}
where the function $R(\widetilde\mu)$ is bounded,
\begin{align*}
\big|R(\widetilde\mu)\big|\le \frac 1 {\widetilde\mu}
 \int_{0}^{\infty}
\frac
{2e^{   -2\pi s  }}
{ 
\big|(is/\widetilde\mu)^{2H+1}-1\big|
}
ds.
\end{align*}
Plugging these into \eqref{phi3} and normalizing by \eqref{cn} gives 
\begin{align*}
\widetilde\varphi_n^{3,a}(t) \simeq & 
 (-1)^n\sin (\pi (\widetilde \gamma_H-\gamma_H))\int_0^{\infty}    f_1(u)
\widetilde g_1\big(\widetilde \mu_n \pi(1-t) u\big)du 
+ \\
&
(-1)^n\cos \pi (\widetilde \gamma_H-\gamma_H)
\int_0^{\infty}    f_1(u) e^{-    \widetilde \mu_n\pi(1-t) u} du + n^{-1} \widetilde r_n^{(3)}(t)
\end{align*}
where we defined 
\begin{equation}\label{psifla}
\widetilde g_1 (x) =
%- \frac  {\sqrt{2H+1}}{\pi}\Re\left\{  i \int_{0}^{\infty}
%\frac
%{e^{-    is x}}
%{ 
%(is)^{2H+1}-1
%}ds
%\right\} =
\frac  {2H+1}{\pi}\Im\left\{    \int_{0}^{\infty}
\frac
{e^{-    is x}}
{ 
(is)^{2H+1}-1
}ds
\right\}.
\end{equation}

\end{proof}

The expression in \eqref{general} can be simplified using the structure, specific to
the fBm. Tracing back the definitions in \cite[Theorem 2.1]{ChK}, we have 
\begin{equation}\label{f1u}
f_1(u) = \frac{\sqrt{2H+1}}{2\pi i}X_0(-u) \Big(\frac 1 {\Lambda_0^+(u)}-\frac 1 {\Lambda_0^-(u)}\Big),
\end{equation}
where $\Lambda_0^\pm(u)$ are the limits of the function
$$
\Lambda_0(z) = z + z^{\alpha-2} \begin{cases}
e^{\frac{1-\alpha} 2 \pi i} & \arg(z)\in (0,\pi) \\
e^{-\frac{1-\alpha} 2 \pi i} & \arg(z)\in (-\pi,0)
\end{cases}
$$
as $z\to u\in \Real$ in the upper and the lower half-planes and 
$$
X_0(z) = \exp \left(\frac 1 \pi \int_0^\infty \frac{\arg(\Lambda_0^+(u))}{t-z}dt\right), \quad z\in \mathbb{C}\setminus \Real_+.
$$
The latter function satisfies 
\begin{equation}\label{X0}
\big|X_0(i)\big| = \sqrt{\frac{2H+1}{2}} \quad \text{and}\quad \arg(X_0(i)) = \frac{2H-1}{8}\pi.
\end{equation}

\begin{lem}
For the function $f_1(u)$, defined in \eqref{f1u}, 
$$
\int_{0}^{\infty} e^{-   \widetilde \nu_n (1-t)u }  \widetilde f_1(u)   du +  
\int_0^{\infty}   \widetilde g_1\big(\widetilde \nu_n (1-t) u\big) f_1(u)  du =0.
$$
\end{lem}
\begin{proof}
For $\widetilde g(\cdot)$ as in \eqref{psifla},
$$
I(t):=  \int_0^{\infty}   \widetilde g_1\big(\widetilde \nu_n (1-t) u\big) f_1(u)  du = 
\frac  {2H+1}{\pi} \,
\Im\left(
\int_0^\infty \frac {J\big(s \widetilde \nu_n (1-t)\big)} {(is)^{2H+1}-1}
ds
\right),
$$
where we defined 
\begin{align*}
J(a) := & \int_0^\infty e^{-    i a u} f_1(u)  du = \\
&
\frac{\sqrt{2H+1}}{2\pi i} \int_0^\infty e^{-    i a u} X_0(-u) \Big(\frac 1 {\Lambda_0^+(u)}-\frac 1 {\Lambda_0^-(u)}\Big)  du, \quad a\in \Real_+.
\end{align*}
Summing up the integrals of the functions 
$$
H_1(z) = e^{iaz} \frac{X_0(z)}{\Lambda_0(z)}\quad \text{and}\quad H_2(z) = e^{-iaz} \frac{X_0(-z)}{\Lambda_0(-z)}
$$
over semicircular contours in the upper and the lower half-planes respectively, the following formula is obtained by the 
usual residue calculus 
$$
J(a) = \Res\big\{H_1; i\big\} + \Res\big\{H_2; -i\big\} = 2e^{-a} \frac {X_0(i)}{\Lambda_0'(i)}.
$$
Plugging this and the explicit formulas \eqref{X0}, we get
$$
I(t) = \sqrt{2} \frac {2H+1}\pi \int_0^\infty e^{-i s\widetilde \nu_n (1-t)} \Im\left(\frac{e^{ \frac {2H-1}{8} \pi i}}{(is)^{2H+1}-1}\right)ds.
$$
In view of the formula \eqref{tildef1}, the claimed equality follows, since $\eta_H+\gamma_H = -\frac 1 2 - \frac {2H-1}{8}$.
\end{proof}

Finally, it is left to check that the eigenfunctions of the bridge are asymptotic to the expressions found in Lemma \ref{lem3.3}: 

\begin{lem}
For any $H\in (0,1)$, 
$$
\left| \frac{\widetilde\varphi_n(t)}{\|\widetilde\varphi_n\|} - \frac {\widetilde \varphi^a_n(t)}{\|\widetilde \varphi^a_n\|} \right|
 \le C n^{-1}\log n, \qquad t\in [0,1]
$$
for some constant $C$.
\end{lem}

\begin{proof}
Denote by $\varphi^a_n(t)$ the leading asymptotic term in the eigenfunctions approximation \eqref{phin} for the base process, 
\begin{equation}\label{vaprhibnd}
\big|\varphi^a_n(t)-\varphi_n(t)\big|= |r_n(t)|n^{-1}\le C n^{-1}
\end{equation}
with a constant $C$. Then by \eqref{phimu} and \eqref{cn},
\begin{align*}
\frac{\big|\widetilde \varphi_n(t)-\widetilde \varphi^a_n(t)\big|}{\widetilde c_n}\, \lesssim \;
&  \widetilde \mu_n^{2H }
\left|
\sum_{k=1}^\infty \frac{\varphi_k(1) \varphi_k(t) }{\mu_k^{2H+1}-\widetilde \mu_n^{2H+1}}
-
 \sum_{k=1}^\infty \frac{\varphi_k^a(1) \varphi_k^a(t) }{(k+\gamma_H)^{2H+1}-\widetilde \mu_n^{2H+1}}
\right| \lesssim\\
&
 \widetilde \mu_n^{2H }
\sum_{k=1}^\infty \Bigg|\frac{  \varphi_k(t) -\varphi_k^a(t) }{ \mu_k^{2H+1}-\widetilde \mu_n^{2H+1} }\Bigg| +
 \widetilde \mu_n^{2H }
\sum_{k=1}^\infty \Bigg|\frac{\varphi_k(1) -\varphi_k^a(1)}{ \mu_k^{2H+1}-\widetilde \mu_n^{2H+1} }\Bigg| +\\
& 
 \widetilde \mu_n^{2H }
\sum_{k=1}^\infty 
\Bigg|
\frac
{\mu_k^{2H+1}-(k+\gamma_H)^{2H+1}}
{\Big(\mu_k^{2H+1}-\widetilde \mu_n^{2H+1}\Big)\Big((k+\gamma_H)^{2H+1}-\widetilde \mu_n^{2H+1}\Big)}
\Bigg|.
\end{align*}
By calculations as in the proof of Lemma \ref{lem2}, the last two terms on the right hand side are shown to  be of order $O(n^{-1})$. 
The first term is of order $O(n^{-1}\log n)$, where $\log n$ factor is due to $O(n^{-1})$ bound 
in \eqref{vaprhibnd}. 
The claim now follows since $\|\widetilde \varphi^a_n\|/\|\widetilde c_n\| = 1+ O(n^{-1})$.

\end{proof}

\section*{Acknowledgement} We are grateful to A.I.Nazarov for bringing our attention to several related works, including 
\cite{Su72} and \cite{N09b}. This work has been supported supported by the ISF grant 558/13.

%\bibliographystyle{imsart-number}
%\bibliography{/Users/Pavel/Dropbox/Pasha_Masha/bibliography/fBm}

\def\cprime{$'$} \def\cprime{$'$} \def\cydot{\leavevmode\raise.4ex\hbox{.}}
  \def\cprime{$'$} \def\cprime{$'$} \def\cprime{$'$}

\end{document}